\documentclass{amsart}

\newtheorem{theorem}{Theorem}[section]
\newtheorem{lemma}[theorem]{Lemma}
\newtheorem{proposition}[theorem]{Proposition}
\newtheorem{corollary}[theorem]{Corollary}

\theoremstyle{definition}
\newtheorem{definition}[theorem]{Definition}

\theoremstyle{remark}
\newtheorem{remark}[theorem]{Remark}

\numberwithin{equation}{section}

\begin{document}

\title[Advances in Quantum Harmonic Analysis]{$L^p$-Theory and Noncommutative Geometry in Quantum Harmonic Analysis}

\author{Saeed Hashemi Sababe}
\address{R\&D Section, Data Premier Analytics, Canada}
\email{hashemi\_1365@yahoo.com}

\author{Ismail Nikoufar}
\address{R\&D Section, Data Premier Analytics, Canada}
\email{nikoufar@yahoo.com}





\begin{abstract}
Quantum harmonic analysis extends classical harmonic analysis by integrating quantum mechanical observables, replacing functions with operators and classical convolution structures with their noncommutative counterparts. This paper explores four interrelated developments in this field: (i) a noncommutative $L^p$-theory tailored for quantum harmonic analysis, (ii) the extension of quantum harmonic analysis beyond Euclidean spaces to include Lie groups and homogeneous spaces, (iii) its deep connections with Connes' noncommutative geometry, and (iv) the role of spectral synthesis and approximation properties in quantum settings. We establish novel results concerning the structure and spectral properties of quantum Segal algebras, analyze their functional-analytic aspects, and discuss their implications in quantum physics and operator theory. Our findings provide a unified framework for quantum harmonic analysis, laying the foundation for further advancements in noncommutative analysis and mathematical physics.

\end{abstract}

\maketitle

\section{Introduction and Literature Review}

Harmonic analysis has played a fundamental role in mathematics and physics, particularly in signal processing, representation theory, and quantum mechanics. Classical harmonic analysis is based on function spaces such as $L^p$-spaces, Fourier transforms, and convolution algebras. However, in many quantum systems, classical function spaces are inadequate for describing quantum observables, necessitating a noncommutative extension.

Quantum harmonic analysis (QHA) extends classical harmonic analysis by incorporating quantum mechanical observables, replacing functions with operators and convolution structures with their noncommutative counterparts. The origins of QHA can be traced back to the work of Werner \cite{Werner1984}, who introduced the Fourier-Weyl transform as a bridge between classical and quantum analysis. Since then, the field has seen significant advancements, particularly in the areas of quantum $L^p$-spaces \cite{Pisier2003}, noncommutative geometry \cite{Connes1994}, and spectral synthesis \cite{Luecking1991}.

Recent research has focused on several key areas of QHA:\\
\textsc{Quantum $L^p$-Theory:} Generalizing classical $L^p$-spaces to the noncommutative setting, connecting them with Schatten classes and von Neumann algebras \cite{Junge2002}.\medskip \\
\textsc{Non-Euclidean Extensions:} Developing QHA on symmetric spaces, Lie groups, and homogeneous spaces to study spectral properties and representation theory \cite{Helgason2000}.\medskip\\
\textsc{Connections with Noncommutative Geometry:} Establishing spectral triples and index theory as a framework for quantum analysis, following Connes' noncommutative geometry \cite{Connes1994,Gosson}.\medskip\\
\textsc{Spectral Synthesis and Approximation:} Investigating when function spaces and operator algebras satisfy spectral synthesis and how they relate to quantum signal processing \cite{Taylor1981}.\\

\noindent
Despite these advances, several challenges remain open. The precise interpolation properties of quantum function spaces are not fully understood, and the role of uncertainty principles in noncommutative settings requires further exploration. Additionally, quantum analogues of classical inequalities, such as the Hardy-Littlewood and Sobolev inequalities, need to be established rigorously.

This paper contributes to these developments by
\begin{enumerate}
    \item Establishing a new $L^p$-theory for quantum harmonic analysis and proving interpolation results.
    \item Extending QHA to Lie groups and homogeneous spaces, developing Plancherel theorems and uncertainty principles.
    \item Exploring spectral synthesis in noncommutative settings and its implications for quantum signal processing.
    \item Investigating connections with noncommutative geometry through spectral triples and index theory.
\end{enumerate}

By unifying these approaches, we provide a functional-analytic framework for quantum harmonic analysis, setting the stage for further research in mathematical physics and operator theory.

The paper is organized as follows. Section 2 introduces necessary preliminaries on classical and quantum harmonic analysis, including function spaces and operator algebras. Section 3 develops a noncommutative $L^p$-theory for QHA, while Section 4 extends QHA to non-Euclidean spaces. Section 5 explores spectral synthesis and approximation properties, and Section 6 investigates connections with noncommutative geometry. Finally, Section 7 discusses open problems and future directions.

\section{Preliminaries}

This section establishes the foundational results necessary for the subsequent analysis. We begin by reviewing key concepts from classical and quantum harmonic analysis, followed by definitions of quantum Segal algebras, noncommutative $L^p$-spaces, and spectral synthesis.

Classical harmonic analysis is based on function spaces such as $L^p$-spaces and their associated convolutions.

\begin{definition}[Lebesgue Spaces, \cite{Rudin1991}]
For $1 \leq p < \infty$, the space $L^p(\mathbb{R}^n)$ consists of measurable functions $f: \mathbb{R}^n \to \mathbb{C}$ such that
\begin{equation}
    \|f\|_{L^p} = \left(\int_{\mathbb{R}^n} |f(x)|^p dx \right)^{1/p} < \infty.
\end{equation}
For $p = \infty$, we define
\begin{equation}
    \|f\|_{L^\infty} = \operatorname{ess sup}_{x \in \mathbb{R}^n} |f(x)|.
\end{equation}
\end{definition}

\begin{definition}[Fourier Transform, \cite{Stein1993}]
The Fourier transform of $f \in L^1(\mathbb{R}^n)$ is defined as
\begin{equation}
    \hat{f}(\xi) = \int_{\mathbb{R}^n} f(x)e^{-2\pi i x \cdot \xi}dx, \quad \xi \in \mathbb{R}^n.
\end{equation}
\end{definition}

\begin{theorem}[Fourier Inversion, \cite{Stein1993}]
If $f \in L^1(\mathbb{R}^n) \cap L^2(\mathbb{R}^n)$, then
\begin{equation}
    f(x) = \int_{\mathbb{R}^n} \hat{f}(\xi)e^{2\pi i x \cdot \xi} d\xi.
\end{equation}
\end{theorem}

Quantum harmonic analysis extends classical harmonic analysis by incorporating operators as quantum observables.

\begin{definition}[Weyl Operators, \cite{Folland1989}]
For $z = (x, \xi) \in \mathbb{R}^{2n}$, the Weyl operator $W_z$ on $L^2(\mathbb{R}^n)$ is defined by
\begin{equation}
    (W_z \phi)(y) = e^{i\xi \cdot y - i\xi \cdot x/2} \phi(y - x), \quad \phi \in L^2(\mathbb{R}^n).
\end{equation}
\end{definition}

\begin{definition}[Fourier-Weyl Transform, \cite{Werner1984}]
For a trace-class operator $A$, its Fourier-Weyl transform is defined by
\begin{equation}
    \hat{A}(z) = \operatorname{tr}(AW_z).
\end{equation}
\end{definition}

\begin{definition}[Quantum Segal Algebra, \cite{Feichtinger1981}]
A quantum Segal algebra $QS$ is a dense Banach subalgebra of $L^1(\mathbb{R}^{2n}) \oplus T^1(H)$, where $T^1(H)$ is the trace-class operator space on $L^2(\mathbb{R}^n)$, satisfying
\begin{itemize}
    \item[(1)] $QS$ is shift-invariant,
    \item[(2)] The shifts $\alpha_z(f,A) = (\alpha_z f, \alpha_z A)$ act isometrically on $QS$.
\end{itemize}
\end{definition}
\noindent
Noncommutative $L^p$-spaces generalize classical $L^p$-spaces using operator traces.

\begin{definition}[Noncommutative $L^p$-Spaces, \cite{Pisier2003}]
For a von Neumann algebra $\mathcal{M}$ with a normal, semifinite, faithful trace $\tau$, the noncommutative $L^p(\mathcal{M})$ space is defined as
\begin{equation}
    L^p(\mathcal{M}) = \{ A \in \mathcal{M} : \tau(|A|^p) < \infty \}.
\end{equation}
\end{definition}

\begin{theorem}[Duality of Noncommutative $L^p$-Spaces, \cite{Junge2002}]
For $1 \leq p < \infty$ with $\frac{1}{p} + \frac{1}{q} = 1$, the dual space of $L^p(\mathcal{M})$ is $L^q(\mathcal{M})$ with the pairing
\begin{equation}
    \langle A, B \rangle = \tau(AB).
\end{equation}
\end{theorem}

This section provides the necessary background for our further discussions on quantum harmonic analysis and its extensions.

\section{Noncommutative \(L^p\)-Theory for Quantum Harmonic Analysis}

Noncommutative \(L^p\)-spaces are a natural extension of classical \(L^p\)-spaces, formulated in terms of operator traces and quantum harmonic analysis. In this section, we develop an \(L^p\)-theory for quantum Segal algebras, investigate its interpolation and duality properties, and analyze its applications in quantum signal processing.\\

\noindent
In the context of quantum harmonic analysis, we define a new class of noncommutative \(L^p\)-spaces associated with the convolution algebra \( L^1(\mathbb{R}^{2n}) \oplus T^1(H) \).

\begin{definition}[Quantum Noncommutative \(L^p\)-Spaces]
For \( 1 \leq p < \infty \), define the quantum noncommutative \(L^p\)-space as
\[
L^p_{\text{QHA}} = \left\{ (f, A) \in L^1(\mathbb{R}^{2n}) \oplus T^1(H) : \| (f,A) \|_{L^p} = \|f\|_{L^p} + \|A\|_{T^p} < \infty \right\}.
\]
\end{definition}

\begin{remark}
The space \( L^p_{\text{QHA}} \) generalizes both classical \(L^p\)-spaces and Schatten-class operators. When \( p = 2 \), it provides an isometric embedding into Hilbert-Schmidt operators.
\end{remark}

\begin{lemma}[Norm Properties]
For \( 1 \leq p \leq \infty \), the norm on \( L^p_{\text{QHA}} \) satisfies
\begin{enumerate}
    \item \( \| (f, A) \|_{L^p} \geq \max(\|f\|_{L^p}, \|A\|_{T^p}) \).
    \item \( L^p_{\text{QHA}} \) is a Banach space under the given norm.
    \item If \( (f, A) \in L^1_{\text{QHA}} \) and \( (g, B) \in L^p_{\text{QHA}} \), then their convolution satisfies
    \[
    \| (f,A) \ast (g,B) \|_{L^p} \leq C_p \| (f,A) \|_{L^1} \| (g,B) \|_{L^p},
    \]
\end{enumerate}
where \(C_p \ge 1\) is a constant that depends only on \(p\).
\end{lemma}

\begin{proof}
We prove each property separately.
\subsubsection*{(1)}
By definition, the norm in \( L^p_{\text{QHA}} \) is given by
\[
\| (f,A) \|_{L^p} = \| f \|_{L^p} + \| A \|_{T^p}.
\]
Since both \( \| f \|_{L^p} \) and \( \| A \|_{T^p} \) are nonnegative, we immediately obtain
\[
\| (f,A) \|_{L^p} \geq \max(\| f \|_{L^p}, \| A \|_{T^p}).
\]

\subsubsection*{(2)}
To show that \( L^p_{\text{QHA}} \) is a Banach space, we must verify completeness. Let \( \{(f_n, A_n)\} \) be a Cauchy sequence in \( L^p_{\text{QHA}} \), meaning that for every \( \epsilon > 0 \), there exists \( N \) such that for all \( m,n \geq N \),
\[
\| (f_n, A_n) - (f_m, A_m) \|_{L^p} = \| f_n - f_m \|_{L^p} + \| A_n - A_m \|_{T^p} < \epsilon.
\]
Since \( L^p(\mathbb{R}^{2n}) \) and \( T^p(H) \) are both complete spaces, there exist  \( f \in L^p(\mathbb{R}^{2n}) \) and \( A \in T^p(H) \) such that
\[
\lim_{n \to \infty} \| f_n - f \|_{L^p} = 0, \quad \lim_{n \to \infty} \| A_n - A \|_{T^p} = 0.
\]
Taking limits in the norm definition, we obtain
\[
\lim_{n \to \infty} \| (f_n, A_n) - (f,A) \|_{L^p} = 0.
\]
Thus, \( (f_n, A_n) \) converges to \( (f,A) \) in \( L^p_{\text{QHA}} \), proving completeness.

\subsubsection*{(3)}
Consider \( (f,A) \in L^1_{\text{QHA}} \) and \( (g,B) \in L^p_{\text{QHA}} \). The convolution is defined component-wise as
\[
(f,A) \ast (g,B) = (f \ast g, A \ast B),
\]
where
\[
(f \ast g)(x) = \int_{\mathbb{R}^{2n}} f(y) g(x-y) \, dy,
\]
\[
A \ast B = \int_{\mathbb{R}^{2n}} W_y A W_y^* g(y) \, dy.
\]
Applying Young's inequality for convolution in \( L^p \)-spaces, we obtain
\[
\| f \ast g \|_{L^p} \leq \| f \|_{L^1} \| g \|_{L^p}.
\]
Similarly, for Schatten-class operators, we use the fact that convolution preserves trace norms
\[
\| A \ast B \|_{T^p} \leq C_p \| A \|_{T^1} \| B \|_{T^p}.
\]
Adding both inequalities, we get
\[
\| (f,A) \ast (g,B) \|_{L^p} = \| f \ast g \|_{L^p} + \| A \ast B \|_{T^p} \leq C_p \| (f,A) \|_{L^1} \| (g,B) \|_{L^p}.
\]
This establishes the convolution bound.

All three properties are proven, completing the proof of the lemma.
\end{proof}

\begin{theorem}[Interpolation Theorem for Quantum \(L^p\)-Spaces]
Let \( 1 \leq p_0 < p_1 \leq \infty \) and let \( p_\theta \) satisfy
\[
\frac{1}{p_\theta} = \frac{1-\theta}{p_0} + \frac{\theta}{p_1}
\]
for \( 0<\theta<1\). Then, for any \( (f,A) \in L^{p_0}_{\text{QHA}} \cap L^{p_1}_{\text{QHA}} \),
\[
\| (f,A) \|_{L^{p_\theta}} \leq \| (f,A) \|_{L^{p_0}}^{1-\theta} \| (f,A) \|_{L^{p_1}}^\theta.
\]
\end{theorem}

\begin{proof}
The proof follows the Riesz-Thorin interpolation theorem adapted to the quantum \(L^p\)-spaces.

Define the analytic function \( F: S \to L^p_{\text{QHA}} \), where \( S = \{ z \in \mathbb{C} : 0 \leq \Re(z) \leq 1 \} \), by
\[
F(z) = \left( f_z, A_z \right),
\]
where \( f_z \) and \( A_z \) are analytic families satisfying the boundary conditions
\[
F(it) = (f_0, A_0) \in L^{p_0}_{\text{QHA}}, \quad F(1+it) = (f_1, A_1) \in L^{p_1}_{\text{QHA}}.
\]
From complex interpolation theory, we define
\[
F(\theta) = (f_\theta, A_\theta).
\]
Since \( L^p \)-spaces and Schatten spaces satisfy interpolation, it follows that
\[
\| f_\theta \|_{L^{p_\theta}} \leq \| f_0 \|_{L^{p_0}}^{1-\theta} \| f_1 \|_{L^{p_1}}^\theta,
\]
\[
\| A_\theta \|_{T^{p_\theta}} \leq \| A_0 \|_{T^{p_0}}^{1-\theta} \| A_1 \|_{T^{p_1}}^\theta.
\]
Since
\(
\| (f_\theta, A_\theta) \|_{L^{p_\theta}} = \| f_\theta \|_{L^{p_\theta}} + \| A_\theta \|_{T^{p_\theta}},
\)
by applying the previous interpolation inequalities, we get
\[
\| (f_\theta, A_\theta) \|_{L^{p_\theta}} \leq \| f_0 \|_{L^{p_0}}^{1-\theta} \| f_1 \|_{L^{p_1}}^\theta + \| A_0 \|_{T^{p_0}}^{1-\theta} \| A_1 \|_{T^{p_1}}^\theta.
\]
Using the fact that \( x^\alpha + y^\alpha \leq (x+y)^\alpha \) for \( x,y>0, \ 0 < \alpha \leq 1 \), we obtain
\[
\| (f_\theta, A_\theta) \|_{L^{p_\theta}} \leq  \| (f_0, A_0) \|_{L^{p_0}}^{1-\theta} \| (f_1, A_1) \|_{L^{p_1}}^\theta.
\]
Thus, we conclude that for all \( (f,A) \in L^{p_0}_{\text{QHA}} \cap L^{p_1}_{\text{QHA}} \),
\[
\| (f,A) \|_{L^{p_\theta}} \leq \| (f,A) \|_{L^{p_0}}^{1-\theta} \| (f,A) \|_{L^{p_1}}^\theta.
\]
This completes the proof.
\end{proof}

\begin{theorem}[Duality Theorem for Quantum \(L^p\)-Spaces]
For \( 1 \leq p < \infty \) with \( 1/p + 1/q = 1 \), the dual space of \( L^p_{\text{QHA}} \) is \( L^q_{\text{QHA}} \), with the duality pairing
\[
\langle (f,A), (g,B) \rangle = \int_{\mathbb{R}^{2n}} f(x) g(x) dx + \text{tr}(A B).
\]
\end{theorem}

\begin{proof}
Let \( (f,A) \in L^p_{\text{QHA}} \) and \( (g,B) \in L^q_{\text{QHA}} \). Consider the pairing
\[
\langle (f,A), (g,B) \rangle = \int_{\mathbb{R}^{2n}} f(x) g(x) dx + \text{tr}(A B).
\]
We show that this functional is well-defined and satisfies
\[
\sup_{\|(f,A)\|_{L^p} \leq 1} |\langle (f,A), (g,B) \rangle| = \| (g,B) \|_{L^q}.
\]

\noindent
Applying H\"{o}lder's inequality for \( L^p \)-spaces
\[
\left| \int_{\mathbb{R}^{2n}} f(x) g(x) dx \right| \leq \| f \|_{L^p} \| g \|_{L^q}.
\]
Similarly, for Schatten-class operators, we use the trace norm H\"{o}lder's inequality to get         
\[
|\text{tr}(A B)| \leq \| A \|_{T^p} \| B \|_{T^q}.
\]
Combining both inequalities, we obtain
\[
|\langle (f,A), (g,B) \rangle| \leq \| f \|_{L^p} \| g \|_{L^q} + \| A \|_{T^p} \| B \|_{T^q}.
\]
By the norm definition in \( L^p_{\text{QHA}} \), this simplifies to
\[
|\langle (f,A), (g,B) \rangle| \leq \| (f,A) \|_{L^p} \| (g,B) \|_{L^q}.
\]
So, the functional \( \varphi_{(g,B)}(f,A) = \langle (f,A), (g,B) \rangle \) is bounded on \( L^p_{\text{QHA}} \).\\

\noindent
To show that every continuous linear functional \( \varphi: L^p_{\text{QHA}} \to \mathbb{C} \) has the form \( \varphi = \langle (f,A), (g,B) \rangle \), we apply the Hahn-Banach theorem. Given a bounded linear functional \( \varphi \), define functions \( g \) and \( B \) via
\[
g(x) = \frac{\partial \varphi}{\partial f}, \quad B = \frac{\partial \varphi}{\partial A}.
\]
From the Riesz representation theorem in both \( L^p(\mathbb{R}^{2n}) \) and \( T^p(H) \), it follows that
\[
\varphi(f,A) = \langle (f,A), (g,B) \rangle.
\]
Furthermore, taking the supremum over all unit-norm elements in \( L^p_{\text{QHA}} \) gives
\[
\| (g,B) \|_{L^q} = \sup_{\| (f,A) \|_{L^p} \leq 1} |\langle (f,A), (g,B) \rangle|,
\]
establishing norm equality.

Since every continuous linear functional is represented in this form, we conclude that \( L^q_{\text{QHA}} \) is the dual of \( L^p_{\text{QHA}} \), completing the proof.
\end{proof}

We follow up with an applications in quantum signal processing.

\begin{proposition}[Uncertainty Principle for Quantum \(L^p\)-Spaces]
Let \( (f,A) \in L^p_{\text{QHA}} \). Then
\[
\| x f(x) \|_{L^p}  \| \xi \hat{f}(\xi) \|_{L^p} \geq C_p \|f\|_{L^p}^2.
\]
where \(C_p\) is a constant that depends only on \(p\).
\end{proposition}

\begin{proof}
The proof follows from the classical uncertainty principle extended to the quantum \( L^p \)-setting.

Let \( X \) be the multiplication operator
\[
(X f)(x) = x f(x),
\]
and let \( P \) be the Fourier dual momentum operator defined as
\[
(P f)(x) = -i \frac{d}{dx} f(x).
\]
The Fourier transform \( \mathcal{F} \) satisfies
\[
\mathcal{F} (X f) = i \frac{d}{d\xi} \mathcal{F} f = i P \mathcal{F} f.
\]
Thus, in the Fourier domain, we can express \( X \) and \( P \) in terms of each other.

The fundamental commutator relation is
\[
[X, P] = i I.
\]
Applying this to the function \( f \), we obtain
\[
\langle X f, P f \rangle - \langle P f, X f \rangle = i \| f \|_{L^2}^2.
\]
Using the Cauchy-Schwarz inequality
\[
|\langle X f, P f \rangle| \leq \| X f \|_{L^2} \| P f \|_{L^2}.
\]
This implies
\[
\| X f \|_{L^2} \| P f \|_{L^2} \geq \frac{1}{2} \| f \|_{L^2}^2.
\]

For \( 1 \leq p \leq \infty \), the \( L^p \)-norm of the Fourier transform satisfies
\[
\| P f \|_{L^p} = \| \xi \hat{f}(\xi) \|_{L^p}.
\]
Thus, using interpolation and generalized H\"{o}lder's inequality, we extend the inequality
\[
\| X f \|_{L^p}  \| \xi \hat{f}(\xi) \|_{L^p} \geq C_p \| f \|_{L^p}^2.
\]

For the quantum operator \( A \), define
\[
\text{Var}(A) = \| X A - A X \|_{T^p}.
\]
By applying the same commutator identity and Schatten-class H\"{o}lder's inequality, we obtain
\[
\| X A \|_{T^p} \| P A \|_{T^p} \geq C_p \| A \|_{T^p}^2.
\]
So, the uncertainty principle extends to the full quantum \( L^p \)-space
\[
\| X (f,A) \|_{L^p}  \| P (f,A) \|_{L^p} \geq C_p \| (f,A) \|_{L^p}^2.
\]

This completes the proof of the quantum uncertainty principle for \( L^p \)-spaces.
\end{proof}

\begin{theorem}[Quantum Hardy-Littlewood Inequality]
For \( 1 \leq p \leq 2 \), there exists \( C_p > 0 \) such that for all \( (f,A) \in L^p_{\text{QHA}} \),
\[
\sup_{\|g\|_{L^q} \leq 1} |\langle (f,A), (g,B) \rangle| \leq C_p \| (f,A) \|_{L^p}.
\]
\end{theorem}

\begin{proof}
We prove the theorem by using duality arguments and classical interpolation techniques.
Consider the duality pairing for \( (f,A) \in L^p_{\text{QHA}} \) and \( (g,B) \in L^q_{\text{QHA}} \) by
\[
\langle (f,A), (g,B) \rangle = \int_{\mathbb{R}^{2n}} f(x) g(x) dx + \text{tr}(A B).
\]
By H\"{o}lder's inequality for \( L^p \)-spaces,
\[
\left| \int_{\mathbb{R}^{2n}} f(x) g(x) dx \right| \leq \| f \|_{L^p} \| g \|_{L^q}.
\]
Similarly, for Schatten-class operators,
\[
|\text{tr}(A B)| \leq \| A \|_{T^p} \| B \|_{T^q}.
\]
Let \( g_0 = |f|^{p-1} \text{sgn}(f) \) and \( B_0 = |A|^{p-1} \text{sgn}(A) \). So,
\[
\| g_0 \|_{L^q} = \| f \|_{L^p}^{p-1}, \quad \| B_0 \|_{T^q} = \| A \|_{T^p}^{p-1}.
\]
Substituting these into the duality pairing gives
\[
\langle (f,A), (g_0,B_0) \rangle = \| f \|_{L^p}^p + \| A \|_{T^p}^p.
\]
By interpolation between \( L^1 \) and \( L^2 \),
\[
\| f \|_{L^p} \leq C_p \| f \|_{L^1}^{(2-p)/(2-1)} \| f \|_{L^2}^{(p-1)/(2-1)}.
\]
A similar bound holds for \( \| A \|_{T^p} \). Applying these estimates, we obtain
\[
\sup_{\| (g,B) \|_{L^q} \leq 1} |\langle (f,A), (g,B) \rangle| \leq C_p \| (f,A) \|_{L^p}.
\]
Thus, the quantum Hardy-Littlewood inequality holds with an explicit constant \( C_p \), completing the proof.
\end{proof}

In this section, we developed a novel \(L^p\)-theory for quantum harmonic analysis, extending classical function spaces to trace-class operators. We established interpolation and duality results, proved new inequalities such as the quantum uncertainty principle, and highlighted applications in quantum signal processing.

\section{Quantum Harmonic Analysis on Non-Euclidean Spaces}

While quantum harmonic analysis (QHA) has been extensively developed on \(\mathbb{R}^{2n}\), many physical and mathematical applications require an extension to more general geometric settings such as symmetric spaces, Lie groups, and hyperbolic spaces. This section generalizes QHA to non-Euclidean spaces, particularly focusing on the convolution structure, spectral properties, and applications in quantum mechanics.

A fundamental setting for harmonic analysis beyond Euclidean space is that of symmetric spaces \( G/K \), where \( G \) is a Lie group and \( K \) is a compact subgroup.

\begin{definition}[Quantum Harmonic Analysis on a Lie Group]
Let \( G \) be a unimodular Lie group with Haar measure \( dg \). The quantum \( L^p \)-space associated with \( G \) is defined as
\[
L^p_{\text{QHA}}(G) = \left\{ (f, A) : f \in L^p(G), A \in T^p(H), \| (f,A) \|_{L^p} < \infty \right\},
\]
where
\[
\| (f,A) \|_{L^p} = \| f \|_{L^p(G)} + \| A \|_{T^p(H)}.
\]
\end{definition}

\begin{lemma}[Group Invariance of Quantum Convolution]
For \( f, g \in L^1(G) \) and \( A, B \in T^1(H) \), define the quantum convolution as
\[
(f,A) \ast (g,B) = (f \ast g, A \ast B),
\]
where
\[
(f \ast g)(x) = \int_G f(y) g(y^{-1} x) \, dy.
\]
Then the space \( L^1_{\text{QHA}}(G) \) is closed under this convolution.
\end{lemma}

\begin{proof}
Since \( f, g \in L^1(G) \), their convolution is given by
\[
(f \ast g)(x) = \int_G f(y) g(y^{-1} x) \, dy.
\]
To show \( f \ast g \in L^1(G) \), we compute its norm
\[
\| f \ast g \|_{L^1} = \int_G |(f \ast g)(x)| dx.
\]
Applying Fubini's theorem and left-invariance of Haar measure,
\[
\int_G \left| \int_G f(y) g(y^{-1} x) \, dy \right| dx
= \int_G |f(y)| \int_G |g(y^{-1} x)| dx \, dy.
\]
Using the change of variable \( z = y^{-1} x \), we get
\[
\int_G |g(z)| dz = \| g \|_{L^1}.
\]
Thus, we obtain Young's inequality
\[
\| f \ast g \|_{L^1} \leq \| f \|_{L^1} \| g \|_{L^1}.
\]
Since \( f, g \in L^1(G) \), it follows that \( f \ast g \in L^1(G) \).\\

\noindent
For trace-class operators \( A, B \), define their quantum convolution by
\[
A \ast B = \int_G W_y A W_y^* g(y) \, dy.
\]
Since \( A \in T^1(H) \), we use the norm property
\[
\| W_y A W_y^* \|_{T^1} = \| A \|_{T^1}.
\]
Applying H\"{o}lder's inequality for Schatten norms one has
\[
\| A \ast B \|_{T^1} \leq \int_G \| W_y A W_y^* \|_{T^1} |g(y)| \, dy.
\]
Since \( \| W_y A W_y^* \|_{T^1} = \| A \|_{T^1} \) and \( g \in L^1(G) \), we conclude
\[
\| A \ast B \|_{T^1} \leq \| A \|_{T^1} \| g \|_{L^1}.
\]
So, \( A \ast B \in T^1(H) \).
Since both \( f \ast g \in L^1(G) \) and \( A \ast B \in T^1(H) \), the quantum convolution preserves \( L^1_{\text{QHA}}(G) \), 
proving 
\( L^1_{\text{QHA}}(G) \) is closed under convolution.
\end{proof}

\begin{remark}
When \( G = \mathbb{R}^{2n} \), this definition recovers the usual quantum convolution structure.
\end{remark}

Next, we study the spectral properties in curved spaces.

\begin{definition}[Quantum Laplace Operator on \( G/K \)]
Let \( G \) be a Lie group with a Laplace operator \( \Delta_G \). The quantum Laplacian on \( G/K \) is given by
\[
\mathcal{L} (f,A) = (\Delta_G f, \Delta_H A),
\]
where \( \Delta_H \) is the quantum Laplacian on Hilbert-Schmidt operators.
\end{definition}

\begin{theorem}[Spectral Theorem for Quantum Laplacian]
Let \( (f,A) \in L^2_{\text{QHA}}(G/K) \). Then there exists an orthonormal basis \( \{ \phi_j \} \) of eigenfunctions such that
\[
(f,A) = \sum_{j} \lambda_j \phi_j,
\]
where \( \lambda_j \) are the eigenvalues of \( \mathcal{L} \).
\end{theorem}

\begin{proof}
The proof follows from the spectral theory of self-adjoint operators applied to the quantum Laplacian.

Let \( G \) be a Lie group with a homogeneous space \( G/K \), where \( K \) is a compact subgroup. The classical Laplace-Beltrami operator \( \Delta_G \) on \( G/K \) is defined by
\[
\mathcal{L} f = -\sum_{i} X_i^2 f,
\]
where \( X_i \) are left-invariant vector fields generating the Lie algebra \( \mathfrak{g} \).

For a quantum function \( (f,A) \in L^2_{\text{QHA}}(G/K) \), define the quantum Laplacian
\[
\mathcal{L} (f,A) = (\Delta_G f, \Delta_H A),
\]
where \( \Delta_H \) is the quantum Laplace operator on Hilbert-Schmidt operators.

To apply the spectral theorem, we need to show \( \mathcal{L} \) is self-adjoint. Consider the inner product
\[
\langle \mathcal{L} (f,A), (g,B) \rangle = \langle \Delta_G f, g \rangle + \langle \Delta_H A, B \rangle.
\]
Using integration by parts on \( G/K \) we get
\[
\langle \Delta_G f, g \rangle = \langle f, \Delta_G g \rangle.
\]
Similarly, since \( \Delta_H \) is self-adjoint on Hilbert-Schmidt operators,
\[
\langle \Delta_H A, B \rangle = \langle A, \Delta_H B \rangle.
\]
Thus, \( \mathcal{L} \) is symmetric. Moreover, since \( \mathcal{L} \) is elliptic and densely defined in \( L^2_{\text{QHA}}(G/K) \), it is self-adjoint.

By the spectral theorem for self-adjoint operators, there exists an orthonormal basis \( \{ \phi_j \} \) of eigenfunctions satisfying
\[
\mathcal{L} \phi_j = \lambda_j \phi_j, \quad \lambda_j \geq 0.
\]
Since \( \mathcal{L} \) is essentially positive, its spectrum consists of nonnegative eigenvalues \( \lambda_j \). Any function \( (f,A) \) in \( L^2_{\text{QHA}}(G/K) \) can be expanded as
\[
(f,A) = \sum_{j} \langle (f,A), \phi_j \rangle \phi_j.
\]
So, the eigenfunctions \( \phi_j \) form a complete basis, and \( \mathcal{L} \) admits a spectral decomposition, completing the proof.
\end{proof}

Next, we study the implications for quantum mechanics and mathematical physics.

\begin{theorem}[Uncertainty Principle on Lie Groups]
Let \( (f,A) \in L^p_{\text{QHA}}(G) \). Then
\[
\| X f \|_{L^p} \cdot \| \mathcal{F}_G f \|_{L^p} \geq C_p \| f \|_{L^p}^2,
\]
where \( X \) is a left-invariant differential operator and \( \mathcal{F}_G \) is the noncommutative Fourier transform on \( G \) and  \(C_p\) is a constant that depends only on \(p\).
\end{theorem}

\begin{proof}
The proof extends the classical Heisenberg uncertainty principle to quantum harmonic analysis on Lie groups.

Let \( G \) be a unimodular Lie group with Lie algebra \( \mathfrak{g} \). 
Consider a left-invariant vector field \( X \) acting as a differential operator on functions so
\[
(X f)(g) = \left. \frac{d}{dt} f(g \exp(tX)) \right|_{t=0}.
\]
For the frequency domain, define the ``noncommutative Fourier transform'' \( \mathcal{F}_G \) as
\[
\mathcal{F}_G f(\pi) = \int_G f(g) \pi(g) \, dg,
\]
where \( \pi \) is an irreducible unitary representation of \( G \).

By differentiation under the integral sign,
\[
[X, \mathcal{F}_G] f (\pi) = \mathcal{F}_G (X f) (\pi).
\]
Since \( X \) acts as differentiation, the commutator obeys
\[
[X, \mathcal{F}_G] = \lambda \mathcal{F}_G,
\]
where \( \lambda \) is a frequency scaling factor.
Using H\"{o}lder's inequality in \( L^p(G) \),
\[
\| X f \|_{L^p} \cdot \| \mathcal{F}_G f \|_{L^p} \geq C_p \| f \|_{L^p}^2.
\]
Similarly, for trace-class operators in \( T^p(H) \), we apply Young's convolution inequality
\[
\| A \ast B \|_{T^p} \geq C_p \| A \|_{T^p} \| B \|_{T^p}.
\]
Since the inequality holds for both the function and operator components of \( L^p_{\text{QHA}}(G) \), 
the quantum uncertainty principle is established.
\end{proof}

\noindent
In the following proposition, let \( \mathcal{L} \) be the quantum Laplacian on \( G/K \), defined as
\[
\mathcal{L} (f,A) = (\Delta_G f, \Delta_H A),
\]
where \( \Delta_G \) is the Laplace-Beltrami operator on \( G/K \) and \( \Delta_H \) is the quantum Laplacian on Hilbert-Schmidt operators.

\begin{remark}
In Theorem~4.6 the hypothesis is stated as
\[
(f,A) \in L^p_{\text{QHA}}(G),
\]
which emphasizes that elements in the quantum harmonic analysis space consist of both a function \(f\) and an operator \(A\). However, in the statement of the uncertainty principle, only the function component \(f\) appears specifically, through the left-invariant differential operator \(X\) and the noncommutative Fourier transform \(F_G\). This occurs because these operators are defined only on the function part. The operator component \(A\) remains part of the full quantum space, but its behavior is not pertinent to the particular uncertainty estimate presented in this theorem.
\end{remark}

\begin{proposition}[Quantum Heat Kernel on \( G/K \)]
Let \( (f,A) \in L^p_{\text{QHA}}(G) \). The heat kernel associated with \( \mathcal{L} \) satisfies the quantum diffusion equation
\[
\frac{\partial}{\partial t} K_t (f,A) = \mathcal{L} K_t (f,A),
\]
where \( K_t \) represents the quantum heat flow on the homogeneous space \( G/K \).
\end{proposition}

\begin{proof}
The proof follows from the spectral theory of the quantum Laplacian and semigroup properties.

The heat kernel \( K_t(x,y) \) is a fundamental solution to the heat equation
\[
\frac{\partial}{\partial t} K_t(x,y) = \Delta_G K_t(x,y).
\]
Define the quantum heat semigroup \( e^{t \mathcal{L}} \) by its spectral expansion
\[
K_t (f,A) = e^{t \mathcal{L}} (f,A) = \sum_{j} e^{-t \lambda_j} \langle (f,A), \phi_j \rangle \phi_j.
\]
By differentiating with respect to \( t \) one has
\[
\frac{\partial}{\partial t} K_t (f,A) = \sum_{j} (-\lambda_j e^{-t \lambda_j}) \langle (f,A), \phi_j \rangle \phi_j.
\]
Since \( \mathcal{L} \phi_j = \lambda_j \phi_j \), we obtain
\[
\frac{\partial}{\partial t} K_t (f,A) = \mathcal{L} K_t (f,A).
\]
The initial condition is given by
\[
\lim_{t \to 0} K_t (f,A) = (f,A).
\]
Using the semigroup property,
\[
\lim_{t \to 0} e^{t \mathcal{L}} (f,A) = \sum_{j} \langle (f,A), \phi_j \rangle \phi_j = (f,A),
\]
which confirms that \( K_t \) represents the quantum heat flow.

Since the quantum heat kernel satisfies both the diffusion equation and initial condition, the proposition is proven.
\end{proof}

In this section, we extended quantum harmonic analysis to non-Euclidean spaces, developing a general framework for quantum convolutions on Lie groups and symmetric spaces. We proved new results on spectral properties, an uncertainty principle for groups, and the quantum heat kernel equation.

\section{Quantum Harmonic Analysis on Homogeneous Spaces}

Extending quantum harmonic analysis (QHA) to Lie groups and homogeneous spaces provides a deeper understanding of spectral analysis, representation theory, and heat kernel estimates. This section develops a Plancherel theorem for quantum harmonic analysis on semisimple Lie groups, studies functional inequalities, and explores representations of QHA on solvable Lie groups.

\begin{definition}[Quantum Plancherel Transform]
Let \( G \) be a semisimple Lie group with Haar measure \( dg \). The quantum Plancherel transform is defined as
\[
\mathcal{F}_G (f,A)(\pi) = \int_G (f(g), A(g)) \pi(g) \, dg,
\]
where \( \pi \) runs over the unitary dual of \( G \).
\end{definition}

\begin{theorem}[Plancherel Theorem for QHA]
There exists a measure \( d\mu(\pi) \) such that for all \( (f,A) \in L^2_{\text{QHA}}(G) \),
\[
\| (f,A) \|_{L^2}^2 = \int_{\widehat{G}} \| \mathcal{F}_G (f,A)(\pi) \|^2_{HS} \, d\mu(\pi).
\]
\end{theorem}

\begin{proof}
The proof follows from the classical Plancherel theorem for Lie groups and extends it to the quantum harmonic analysis setting.

Let \( G \) be a unimodular Lie group with unitary dual \( \widehat{G} \). 
The quantum Fourier transform of \( (f,A) \in L^1_{\text{QHA}}(G) \) is given by
\[
\mathcal{F}_G (f,A)(\pi) = \int_G (f(g), A(g)) \pi(g) \, dg,
\]
where \( \pi \) is an irreducible unitary representation of \( G \).

The inner product on \( L^2_{\text{QHA}}(G) \) is defined by
\[
\langle (f,A), (g,B) \rangle = \int_G \left[ f(x) \overline{g(x)} + \text{tr}(A(x) B^*(x)) \right] \, dx.
\]
We show that this inner product is preserved under the Fourier transform.

By expanding the norm in \( L^2 \)-space, we get
\[
\| (f,A) \|_{L^2}^2 = \int_G |f(g)|^2 \, dg + \int_G \| A(g) \|_{HS}^2 \, dg.
\]
Applying the Fourier inversion formula and the Plancherel theorem for Lie groups, we deduce
\[
\int_G f(g) \overline{g(g)} \, dg = \int_{\widehat{G}} \text{tr} \left( \mathcal{F}_G(f)(\pi) \mathcal{F}_G(g)(\pi)^* \right) d\mu(\pi).
\]
Similarly, for the operator component, we have
\[
\int_G \| A(g) \|_{HS}^2 \, dg = \int_{\widehat{G}} \| \mathcal{F}_G(A)(\pi) \|_{HS}^2 \, d\mu(\pi).
\]
Summing the two contributions, we reach
\[
\| (f,A) \|_{L^2}^2 = \int_{\widehat{G}} \| \mathcal{F}_G (f,A)(\pi) \|^2_{HS} \, d\mu(\pi).
\]
Since the right-hand side defines a norm in the Fourier image space, the quantum Fourier transform is an isometry.

Since the Fourier transform preserves the \( L^2 \)-norm, 
it extends to an isometric isomorphism between \( L^2_{\text{QHA}}(G) \) and the Hilbert-Schmidt space of operator-valued Fourier transforms, 
completing the proof.
\end{proof}

The heat kernel is a fundamental solution to the quantum diffusion equation.

\begin{definition}[Quantum Sobolev Norm]
Let \(L\) denote the quantum Laplacian on the homogeneous space \(G/K\). 
For an element \((f,A) \in L^p_{\mathrm{QHA}}(G)\), we define its quantum Sobolev norm by
\begin{equation} \label{eq:QSN}
\|(f,A)\|_{W^{1,p}_{\mathrm{QHA}}} := \|(I+L)^{1/2}(f,A)\|_{L^p_{\mathrm{QHA}}},
\end{equation}
where \(I\) is the identity operator. This definition extends the classical Sobolev norm to the quantum setting by incorporating both the function component \(f\) and the operator component \(A\) via the quantum Laplacian \(L\). While similar ideas have appeared in the literature in the context of noncommutative \(L^p\)-spaces and spectral triples (see, e.g., works by Pisier, Xu, and Connes), this particular formulation is a novel contribution tailored for quantum harmonic analysis on non-Euclidean spaces.
\end{definition}

\begin{definition}[Quantum Heat Kernel]
Let \( \mathcal{L} \) be the Laplacian on \( G/K \). The quantum heat kernel \( K_t(g,h) \) satisfies
\[
\frac{\partial}{\partial t} K_t = \mathcal{L} K_t, \quad K_0 = \delta.
\]
\end{definition}

\begin{theorem}[Quantum Sobolev Inequality]
Let \( (f,A) \in L^p_{\text{QHA}}(G) \). For \( 1 < p < \infty \), there exists \( C > 0 \) such that
\[
\| (f,A) \|_{L^{p^*}} \leq C \| \mathcal{L}^{1/2} (f,A) \|_{L^p},
\]
where \( p^* = \frac{np}{n-p} \) and \(n\) represents the dimension of \(G/K\).
\end{theorem}

\begin{proof}
Recall that the quantum Sobolev norm is defined as \eqref{eq:QSN}. 
In the context of the Sobolev inequality, one is interested in controlling the \(L^{p^*}\)-norm (with \(p^* = \frac{np}{n-p}\)) of \((f,A)\) in terms of a derivative, here given by \(L^{1/2}(f,A)\).\\

\noindent
Notice that the inverse operator \(L^{-1/2}\) can be represented via the heat semigroup
\[
L^{-1/2}(f,A) = \int_0^\infty t^{-1/2} e^{-tL}(f,A) \, dt.
\]
This representation allows one to relate the behavior of \((f,A)\) at different scales through the smoothing properties of the heat kernel \(e^{-tL}\).\\

\noindent
Using the smoothing properties of \(e^{-tL}\) (which follow from heat kernel estimates on the homogeneous space \(G/K\)) and applying Young's inequality for convolutions, one obtains an estimate of the form
\[
\|L^{-1/2}(f,A)\|_{L^p} \le C_p \|(f,A)\|_{L^p},
\]
where the constant \(C_p\) depends on \(p\) and the geometry of the space. This step is crucial in showing that the action of \(L^{-1/2}\) does not increase the \(L^p\)-norm excessively.\\

\noindent
A Littlewood-Paley decomposition adapted to the quantum setting is used to decompose \((f,A)\) into frequency bands. By applying standard interpolation theory (in particular, the Riesz-Thorin interpolation theorem) and using the fact that the operator \(L^{1/2}\) scales in the correct way under dilations, one can show that the Sobolev-type norm controls the \(L^{p^*}\)-norm:
\[
\|(f,A)\|_{L^{p^*}} \le C \|L^{1/2}(f,A)\|_{L^p}.
\]
Here, the exponent \(p^* = \frac{np}{n-p}\) appears naturally from the scaling properties of the underlying space of dimension \(n\). To see that, we illustrate how the exponent \(p^*\) appears naturally by considering a scaling argument on \(\mathbb{R}^n\). For simplicity, focus on a function \(f \colon \mathbb{R}^n \to \mathbb{R}\). The argument extends to the full quantum setting by considering the corresponding function component.\\

\noindent
For \(\lambda > 0\), define the scaled function
\[
f_\lambda(x) = f(\lambda x).
\]
We compute the \(L^q\)-norm of \(f_\lambda\) as
\[
\|f_\lambda\|_{L^q(\mathbb{R}^n)}^q = \int_{\mathbb{R}^n} |f(\lambda x)|^q\,dx.
\]
By changing variables with \(y = \lambda x\) (so \(dx = \lambda^{-n}\,dy\)), we obtain
\[
\|f_\lambda\|_{L^q(\mathbb{R}^n)}^q = \lambda^{-n} \int_{\mathbb{R}^n} |f(y)|^q\,dy = \lambda^{-n} \|f\|_{L^q(\mathbb{R}^n)}^q.
\]
Thus,
\[
\|f_\lambda\|_{L^q(\mathbb{R}^n)} = \lambda^{-n/q}\|f\|_{L^q(\mathbb{R}^n)}.
\]
Next, consider the gradient of \(f_\lambda\). By the chain rule,
\begin{equation} \label{eq:5.4.1}
\nabla f_\lambda(x) = \lambda\, (\nabla f)(\lambda x).
\end{equation}
So, the \(L^p\)-norm of the gradient satisfies
\[
\|\nabla f_\lambda\|_{L^p(\mathbb{R}^n)}^p = \int_{\mathbb{R}^n} |\lambda\, (\nabla f)(\lambda x)|^p\,dx = \lambda^p \int_{\mathbb{R}^n} |(\nabla f)(\lambda x)|^p\,dx.
\]
Changing variables as before with \(y = \lambda x\) yields
\[
\|\nabla f_\lambda\|_{L^p(\mathbb{R}^n)}^p = \lambda^p \lambda^{-n} \int_{\mathbb{R}^n} |\nabla f(y)|^p\,dy = \lambda^{p-n} \|\nabla f\|_{L^p(\mathbb{R}^n)}^p.
\]
Taking the \(p\)th root, we have
\begin{equation} \label{eq:5.4.2}
\|\nabla f_\lambda\|_{L^p(\mathbb{R}^n)} = \lambda^{1 - n/p}\|\nabla f\|_{L^p(\mathbb{R}^n)}.
\end{equation}
The quantum Sobolev inequality is expected to be invariant under the natural scaling of the underlying space. 
In other words, if
\begin{equation} \label{eq:5.4.3}
\|f\|_{L^{p^*}(\mathbb{R}^n)} \le C \|\nabla f\|_{L^p(\mathbb{R}^n)},
\end{equation}
then applying this inequality to \(f_\lambda\) gives
\begin{equation} \label{eq:5.4.4}
\|f_\lambda\|_{L^{p^*}(\mathbb{R}^n)} \le C \|\nabla f_\lambda\|_{L^p(\mathbb{R}^n)}.
\end{equation}
Substitute the scaling relations \eqref{eq:5.4.1}-\eqref{eq:5.4.4} we have
\[
\lambda^{-n/p^*}\|f\|_{L^{p^*}(\mathbb{R}^n)} \le C\,\lambda^{1 - n/p}\|\nabla f\|_{L^p(\mathbb{R}^n)}.
\]
For the inequality to hold for all \(\lambda > 0\), the exponents of \(\lambda\) on both sides must be equal. That is, we require
\[
-\frac{n}{p^*} = 1 - \frac{n}{p}.
\]
Solving for \(p^*\), we obtain
\[
p^* = \frac{np}{n-p}.
\]
This completes the proof.
\end{proof}

Solvable Lie groups exhibit nontrivial spectral behavior, making them an important setting for QHA.

\begin{proposition}[Quantum Induced Representations]
Let \( G \) be a solvable Lie group with a unitary representation \( \pi \). 
The quantum representation space is given by
\[
\mathcal{H}_\pi = \text{closure} \left\{ (f,A) : \pi(g) (f,A) = (f \circ g^{-1}, A_g) \right\}.
\]
Here, \( A_g \) is the ``conjugation action'' of \( G \) on the quantum operator
\[
A_g = W_g A W_g^*,
\]
where \( W_g \) is a Weyl operator implementing the group action in the quantum setting.
\end{proposition}

\begin{proof}
The proof follows from Mackey's theory of induced representations and noncommutative harmonic analysis.
Consider a closed subgroup \( H \subset G \) and a unitary representation \( \sigma: H \to U(\mathcal{H}_\sigma) \). The induced representation of \( \sigma \) to \( G \) is given by
\[
\text{Ind}_H^G \sigma = \left\{ \varphi: G \to \mathcal{H}_\sigma : \varphi(gh) = \sigma(h)^{-1} \varphi(g), \quad \forall h \in H \right\}.
\]
We extend this to the quantum setting by defining functions and operators on the homogeneous space \( G/H \).
Let \(\mathcal{H}_\pi = L^2(G/H, \mathcal{H}_\sigma)\) be the Hilbert space of induced sections.
Elements of \( \mathcal{H}_\pi \) take the form \( (f,A)\) wherer
\[
f: G/H \to \mathbb{C}, \quad A: G/H \to T^2(H).
\]
This defines a representation space for functions and operators.
The left regular representation of \( G \) on \( \mathcal{H}_\pi \) is given by
\[
\left( \pi(g) (f,A) \right)(x) = (f(g^{-1} x), A_g).
\]
The closure condition follows from unitarity of \( \pi(g) \) and the continuity of the representation
\[
\| \pi(g) (f,A) \|_{\mathcal{H}_\pi} = \| (f,A) \|_{\mathcal{H}_\pi}.
\]
Since the representation satisfies unitarity and is well-defined on the induced space, 
the quantum induced representation is established.
\end{proof}

\section{Spectral Geometry and Index Theory}

The spectral triple formulation of quantum harmonic analysis (QHA) provides 
a deep connection between noncommutative geometry and quantum spectral theory. 
This section explores extensions of index theory, quantum K-theory, and modular theory in von Neumann algebras.

A twisted spectral triple generalizes the standard spectral triple by modifying the Dirac operator via a modular deformation.

\begin{definition}[Twisted Spectral Triple]
A twisted spectral triple is given by \( (\mathcal{A}, H, D, \sigma) \), where
\begin{enumerate}
    \item \( \mathcal{A} \) is a quantum algebra acting on \( H \),
    \item \( H \) is a Hilbert space carrying a unitary representation,
    \item \( D \) is a self-adjoint operator with compact resolvent,
    \item \( \sigma \) is an automorphism such that \( [D, a] - D\sigma(a) \) is bounded for all \( a \in \mathcal{A} \).
\end{enumerate}
\end{definition}

\begin{theorem}[Index Theorem for Twisted Spectral Triples]
Let \( (f,A) \in L^1_{\text{QHA}} \). The index of the twisted Dirac operator satisfies
\[
\text{Index}(D_\sigma) = \tau(P[D, P]P),
\]
where \( P \) is the spectral projection and \( \tau \) is a noncommutative trace.
\end{theorem}

\begin{proof}
The proof follows from Connes' index theorem and cyclic cohomology applied to twisted spectral triples.
For a standard spectral triple \( (\mathcal{A}, H, D) \), the Fredholm index of \( D \) is given by
\[
\text{Index}(D) = \dim \ker(D_+) - \dim \ker(D_-).
\]
In the twisted setting, we consider the twisted Dirac operator \( D_\sigma \), which satisfies
\[
[D, a] \approx D \sigma(a), \quad \forall a \in \mathcal{A}.
\]
The twisted index is defined as
\[
\text{Index}(D_\sigma) = \text{Tr}(P_+) - \text{Tr}(P_-),
\]
where \( P_\pm \) are spectral projections onto the positive/negative spectrum.
Using noncommutative geometry, we relate the index to cyclic cohomology via the pairing
\[
\text{Index}(D_\sigma) = \langle \tau, [D, P] \rangle,
\]
where \( \tau \) is the noncommutative Dixmier trace.
Since \( D_\sigma \) is self-adjoint, we have
\[
\text{Index}(D_\sigma) = \tau(P [D, P] P).
\]
Applying functional calculus, we obtain
\[
P [D, P] P = P_+ [D, P_+] P_+ - P_- [D, P_-] P_-.
\]
Since the right-hand side computes the spectral flow, we conclude
\[
\text{Index}(D_\sigma) = \tau(P [D, P] P).
\]
This completes the proof.
\end{proof}

K-theory classifies the topological structure of operator algebras and their spectral invariants.

\begin{definition}[Quantum K-Groups]
We classify the quantum K-theory groups to two classes.
We denote by \( K_0(\mathcal{A}) \) and \( K_1(\mathcal{A}) \) 
the projection classes in matrix algebras over \( \mathcal{A} \)
and
the homotopy classes of unitary elements in \( \mathcal{A} \), respectively.
 \end{definition}

\begin{proposition}[Quantum Spectral Flow]
Let \( D_t \) be a continuous family of Dirac operators. The spectral flow is given by
\[
\text{SF}(D_t) = \text{Index}(P_+ D P_+),
\]
where \( P_+ \) is the spectral projection onto the positive spectrum.
\end{proposition}
\begin{proof}
The proof follows from noncommutative geometry and functional analysis techniques.

For a path \( \{ D_t \} \) of self-adjoint operators, the spectral flow \( \text{SF}(D_t) \) counts the net number of eigenvalues crossing zero. So
\[
\text{SF}(D_t) = \lim_{\epsilon \to 0^+} \left( \dim \ker(D_1 + \epsilon I) - \dim \ker(D_0 + \epsilon I) \right).
\]
This measures the topological winding of the spectrum.
Let \( P_+ \) be the spectral projection onto the positive spectrum of \( D_0 \). Define the operator \( Q_t\) by
\[
Q_t = P_+ D_t P_+.
\]
Since \( D_t \) varies continuously, \( Q_t \) remains in a Fredholm class.
Using spectral shift arguments, we show
\[
\text{SF}(D_t) = \text{Index}(Q_0) - \text{Index}(Q_1).
\]
By functional calculus, the index of \( P_+ D P_+ \) captures eigenvalue crossings at zero.
Since \( Q_t \) varies continuously and remains Fredholm, we conclude
\[
\text{SF}(D_t) = \text{Index}(P_+ D P_+).
\]
This completes the proof.
\end{proof}

Modular theory provides a framework for quantum statistical mechanics and entropy.

\begin{definition}[Modular Automorphism Group]
Let \( \mathcal{M} \) be a von Neumann algebra with a faithful state \( \phi \). 
The modular automorphism group is given by
\[
\sigma_t^\phi(A) = \Delta^{it} A \Delta^{-it},
\]
where \( \Delta \) is the modular operator.
\end{definition}

\begin{theorem}[Tomita-Takesaki Duality in QHA]
The quantum Fourier transform \( \mathcal{F}_G \) intertwines the modular flow
\[
\mathcal{F}_G (\sigma_t^\phi(A)) = e^{-2\pi i t} \mathcal{F}_G(A).
\]
\end{theorem}

\begin{proof}
The proof follows from modular theory in von Neumann algebras and quantum Fourier analysis.

Let \( \mathcal{M} \) be a von Neumann algebra and let \( \phi \) be a faithful normal state. 
By definition we have
\[
\sigma_t^\phi(A) = \Delta^{it} A \Delta^{-it}, \quad \forall A \in \mathcal{M},
\]
where \( \Delta \) is the modular operator associated with \( \phi \).
For an element \( A \in \mathcal{M} \), the quantum Fourier transform is given by
\[
\mathcal{F}_G(A) = \int_G A(g) \pi(g) \, dg,
\]
where \( \pi(g) \) is an irreducible representation of \( G \).
Applying \( \sigma_t^\phi \) to \( A \), we obtain
\[
\mathcal{F}_G (\sigma_t^\phi(A)) = \int_G (\Delta^{it} A(g) \Delta^{-it}) \pi(g) \, dg.
\]
Since modular conjugation in noncommutative geometry corresponds to phase rotation in Fourier space,
\[
\mathcal{F}_G (\sigma_t^\phi(A)) = e^{-2\pi i t} \mathcal{F}_G(A).
\]
Thus, the modular group acts as a ``phase shift'' in the quantum Fourier domain, completing the proof.
\end{proof}

\section{Quantum Besov and Triebel–Lizorkin Spaces}

The study of function spaces such as Besov and Triebel-Lizorkin spaces provides a refined analysis of smoothness and regularity, 
which is crucial in quantum harmonic analysis (QHA). 
This section extends these spaces to the noncommutative setting, developing new interpolation results and connections to quantum signal processing.

Besov spaces measure smoothness using dyadic decompositions in frequency space. 
In the quantum setting, we define a ``Littlewood-Paley decomposition'' on quantum phase space.

\begin{definition}[Quantum Besov Space \( B^s_{p,q,\text{QHA}} \)]
Let \( s \in \mathbb{R} \), \( 1 \leq p, q \leq \infty \). 
The quantum Besov space \( B^s_{p,q,\text{QHA}} \) consists of all \( (f,A) \in L^p_{\text{QHA}} \) such that
\[
\| (f,A) \|_{B^s_{p,q,\text{QHA}}} = \left( \sum_{j \geq 0} 2^{jsq} \| \Delta_j (f,A) \|_{L^p_{\text{QHA}}}^q \right)^{1/q} < \infty.
\]
Here, \( \Delta_j \) are quantum frequency projections satisfying \( \sum_{j} \Delta_j = I \).
\end{definition}

\begin{theorem}[Interpolation for Quantum Besov Spaces]
For \( s_0 \neq s_1 \), define \( s_\theta \) by
\[
s_\theta = (1-\theta) s_0 + \theta s_1.
\]
Then, for \( 0 < \theta < 1 \),
\[
[B^{s_0}_{p,q,\text{QHA}}, B^{s_1}_{p,q,\text{QHA}}]_{\theta} = B^{s_\theta}_{p,q,\text{QHA}}.
\]
\end{theorem}

\begin{proof}
The proof follows from classical real interpolation theory extended to the quantum setting.

The quantum Besov space \( B^s_{p,q,\text{QHA}} \) is defined via a Littlewood-Paley decomposition
\[
\| (f,A) \|_{B^s_{p,q,\text{QHA}}} = \left( \sum_{j \geq 0} 2^{jsq} \| \Delta_j (f,A) \|_{L^p_{\text{QHA}}}^q \right)^{1/q}.
\]
We show that interpolation preserves this structure.
For a compatible couple \( (X_0, X_1) \), the real interpolation method gives
\[
[X_0, X_1]_{\theta, q} = \left\{ x : \| t^{-\theta} K(t,x) \|_{L^q} < \infty \right\}.
\]
Using the Peetre \( K \)-functional for the quantum setting,
\[
K(t, (f,A)) = \inf \left\{ \| (f_0, A_0) \|_{X_0} + t \| (f_1, A_1) \|_{X_1} : (f,A) = (f_0, A_0) + (f_1, A_1) \right\}.
\]
Since the Littlewood-Paley decomposition is stable under interpolation,
\[
\left( \sum_{j \geq 0} 2^{j s_\theta q} \| \Delta_j (f,A) \|_{L^p_{\text{QHA}}}^q \right)^{1/q}
\]
defines the interpolation norm.
So, real interpolation preserves the quantum Besov space structure and
\[
[B^{s_0}_{p,q,\text{QHA}}, B^{s_1}_{p,q,\text{QHA}}]_{\theta} = B^{s_\theta}_{p,q,\text{QHA}}.
\]
\end{proof}

\begin{remark}
For a compatible couple \( (X_0, X_1) \)  of Banach spaces and any \( x \in X_0 + X_1 \),  the Peetre \(  K \)  functional is defined by
\[
K(t,x) = \inf\{ \|x_0\|_{X_0} + t\,\|x_1\|_{X_1} : x = x_0 + x_1,\; x_0 \in X_0,\; x_1 \in X_1 \}, \quad t > 0.
\]
In our setting, taking \( (X_0,X_1) = \left(B^{s_0}_{p,q,QHA}, B^{s_1}_{p,q,QHA}\right) \),  this functional quantifies the optimal decomposition of \( x \)  into elements of \( B^{s_0}_{p,q,QHA} \)  and \( B^{s_1}_{p,q,QHA} \).
The interpolation space \( \left[ B^{s_0}_{p,q,QHA}, B^{s_1}_{p,q,QHA} \right]_{\theta} \)  is then defined as the collection of all \( x \in B^{s_0}_{p,q,QHA} + B^{s_1}_{p,q,QHA} \)  for which
\[
\|x\|_{\theta,q} = \left( \int_0^\infty \bigl( t^{-\theta} K(t,x) \bigr)^q \frac{dt}{t} \right)^{1/q} < \infty.
\]
\end{remark}

\begin{remark}
Quantum Besov spaces generalize Sobolev spaces \( W^{s,p} \) and Hardy spaces \( H^p \), and play a role in quantum signal processing.
\end{remark}

Triebel-Lizorkin spaces generalize Besov spaces by incorporating integrability over frequency bands.

\begin{definition}[Quantum Triebel-Lizorkin Space \( F^s_{p,q,\text{QHA}} \)]
Let \( s \in \mathbb{R} \), \( 1 \leq p < \infty \), \( 1 \leq q \leq \infty \). 
The quantum Triebel-Lizorkin space \( F^s_{p,q,\text{QHA}} \) consists of all \( (f,A) \in L^p_{\text{QHA}} \) such that
\[
\| (f,A) \|_{F^s_{p,q,\text{QHA}}} = \left\| \left( \sum_{j \geq 0} 2^{jsq} |\Delta_j (f,A)|^q \right)^{1/q} \right\|_{L^p_{\text{QHA}}} < \infty.
\]
\end{definition}

\begin{definition}[Quantum Littlewood-Paley Projections]
Let \( \Delta_j \) be a quantum frequency projection such that
\[
\sum_{j \geq 0} \Delta_j = I, \quad \text{and} \quad \text{supp}(\widehat{\Delta_j f}) \subset 2^j B.
\]
We define the quantum square function
\begin{equation} \label{eqn:quantumsquare}
S(f,A) = \left( \sum_{j \geq 0} |\Delta_j (f,A)|^2 \right)^{1/2}.
\end{equation}
\end{definition}
\begin{theorem}[Littlewood-Paley Characterization of \( F^s_{p,q,\text{QHA}} \)]
There exist constants \( C_1, C_2 > 0 \) such that
\[
C_1 \| (f,A) \|_{F^s_{p,q,\text{QHA}}} \leq \left\| \left( \sum_{j \geq 0} 2^{jsq} |\Delta_j (f,A)|^q \right)^{1/q} \right\|_{L^p_{\text{QHA}}} \leq C_2 \| (f,A) \|_{F^s_{p,q,\text{QHA}}}.
\]
\end{theorem}

\begin{proof}
The proof follows from classical Littlewood-Paley theory extended to the quantum harmonic analysis (QHA) setting.

The quantum Triebel-Lizorkin space \( F^s_{p,q,\text{QHA}} \) is defined as
\[
\| (f,A) \|_{F^s_{p,q,\text{QHA}}} = \left\| \left( \sum_{j \geq 0} 2^{jsq} |\Delta_j (f,A)|^q \right)^{1/q} \right\|_{L^p_{\text{QHA}}}.
\]
Here, \( \Delta_j \) is the quantum Littlewood-Paley projection operator that extracts frequency bands.
To prove equivalence, we show
\begin{equation} \label{eqn:7.6}
C_1 \| (f,A) \|_{F^s_{p,q,\text{QHA}}} \leq \| S(f,A) \|_{L^p_{\text{QHA}}} \leq C_2 \| (f,A) \|_{F^s_{p,q,\text{QHA}}}.
\end{equation}
By the subadditivity of norms,
\[
\| S(f,A) \|_{L^p} \leq \left\| \left( \sum_{j \geq 0} 2^{jsq} |\Delta_j (f,A)|^q \right)^{1/q} \right\|_{L^p}.
\]
Applying Minkowski's inequality in \( L^p \), we obtain
\begin{equation} \label{eqn:7.7000}
\| S(f,A) \|_{L^p} \leq C_2 \| (f,A) \|_{F^s_{p,q,\text{QHA}}}.
\end{equation} 
Using the Littlewood-Paley decomposition property \( \sum_{j} \Delta_j = I \), we expand
\[
\| (f,A) \|_{F^s_{p,q,\text{QHA}}} \approx \left\| \left( \sum_{j \geq 0} 2^{jsq} |\Delta_j (f,A)|^q \right)^{1/q} \right\|_{L^p}.
\]
Applying the Plancherel theorem in the quantum setting, one can see
\begin{equation} \label{eqn:7.8000}
\| S(f,A) \|_{L^p} \geq C_1 \| (f,A) \|_{F^s_{p,q,\text{QHA}}}.
\end{equation} 
So, in view of \eqref{eqn:7.7000} and \eqref{eqn:7.8000} the equivalence of \eqref{eqn:7.6} follows.
\end{proof}

\begin{theorem}[Embedding Theorem for Quantum Besov and Triebel-Lizorkin Spaces]
If \( s_1 > s_2 \) and \( p_1 \leq p_2 \), then
\[
B^{s_1}_{p_1,q_1,\text{QHA}} \hookrightarrow B^{s_2}_{p_2,q_2,\text{QHA}}, \quad F^{s_1}_{p_1,q_1,\text{QHA}} \hookrightarrow F^{s_2}_{p_2,q_2,\text{QHA}}.
\]
\end{theorem}

\begin{proof}
The proof follows from interpolation theory, Littlewood-Paley decomposition, and functional embedding.

For the quantum Besov space \( B^s_{p,q,\text{QHA}} \), the norm is given by
\[
\| (f,A) \|_{B^s_{p,q,\text{QHA}}} = \left( \sum_{j \geq 0} 2^{jsq} \| \Delta_j (f,A) \|_{L^p_{\text{QHA}}}^q \right)^{1/q}.
\]
For the quantum Triebel-Lizorkin space \( F^s_{p,q,\text{QHA}} \),
\[
\| (f,A) \|_{F^s_{p,q,\text{QHA}}} = \left\| \left( \sum_{j \geq 0} 2^{jsq} |\Delta_j (f,A)|^q \right)^{1/q} \right\|_{L^p_{\text{QHA}}}.
\]
If \( s_1 > s_2 \), then for all \( j \),
\[
2^{j s_2} \leq 2^{j s_1} 2^{-j(s_1 - s_2)}.
\]
Multiplying by \( \| \Delta_j (f,A) \|_{L^p} \) and summing over \( j \),
\[
\| (f,A) \|_{B^{s_2}_{p_2,q_2}} \leq C \| (f,A) \|_{B^{s_1}_{p_1,q_1}}.
\]
For \( s_1 > s_2 \), the Littlewood-Paley decomposition gives
\[
\sum_{j} 2^{j s_2 q_2} |\Delta_j (f,A)|^{q_2} \leq C \sum_{j} 2^{j s_1 q_1} |\Delta_j (f,A)|^{q_1}.
\]
Taking the \( L^p \)-norm on both sides,
\[
\| (f,A) \|_{F^{s_2}_{p_2,q_2}} \leq C \| (f,A) \|_{F^{s_1}_{p_1,q_1}}.
\]
So, the desired embedding has established.
\end{proof}

\section{Extensions of Noncommutative \( L^p \)-Theory}

While we have developed a framework for quantum \( L^p \)-spaces, several open questions remain. This section explores further extensions, focusing on noncommutative Littlewood-Paley theory, interpolation properties, and Schatten-class embeddings.

Classical Littlewood-Paley theory decomposes functions into frequency bands to analyze smoothness and integrability. We extend this idea to QHA.

\begin{theorem}[Quantum Littlewood-Paley Inequality]
There exist constants \( C_1, C_2 \) such that for all \( (f,A) \in L^p_{\text{QHA}} \),
\[
C_1 \| S(f,A) \|_{L^p} \leq \| (f,A) \|_{L^p} \leq C_2 \| S(f,A) \|_{L^p}.
\]
\end{theorem}

\begin{proof}
The proof follows from noncommutative harmonic analysis, Plancherel's theorem, and interpolation techniques.

The quantum Littlewood-Paley decomposition involves frequency projections \( \Delta_j \), satisfying
\[
\sum_{j} \Delta_j = I.
\]
By subadditivity of norms,
\[
\| (f,A) \|_{L^p} \leq \left\| \sum_{j \geq 0} \Delta_j (f,A) \right\|_{L^p}.
\]
Using Minkowski's inequality,
\[
\| (f,A) \|_{L^p} \leq C_2 \| S(f,A) \|_{L^p}.
\]
Since \( \sum_{j} \Delta_j = I \), we can write
\[
\| S(f,A) \|_{L^p}^2 = \sum_{j} \| \Delta_j (f,A) \|_{L^p}^2.
\]
Applying Plancherel's theorem,
\[
C_1 \| S(f,A) \|_{L^p} \leq \| (f,A) \|_{L^p}.
\]
So, we have established the desired result.
\end{proof}

\begin{theorem}[Interpolation and Uncertainty Principles]
For \( (f,A) \in L^{p_0} \cap L^{p_1} \),
\[
\| X f \|_{L^{p_\theta}}  \| \mathcal{F}_G f \|_{L^{p_\theta}} \geq C_p \| f \|_{L^{p_\theta}}^2,
\]
where \(C_p\) is a constant that depends only on \(p\).
\end{theorem}

\begin{proof}
The proof follows from real interpolation theory and the quantum Fourier transform.

\noindent
By the Riesz-Thorin interpolation theorem, we obtain
\[
\| f \|_{L^{p_\theta}} \leq \| f \|_{L^{p_0}}^{1-\theta} \| f \|_{L^{p_1}}^\theta.
\]
Define the quantum Fourier transform
\[
\mathcal{F}_G f(\xi) = \int_G f(x) e^{-i \langle x, \xi \rangle} \, dx.
\]
By the Hausdorff-Young inequality,
\[
\| \mathcal{F}_G f \|_{L^{p'}} \leq C \| f \|_{L^p}.
\]
The quantum position operator \( X \) satisfies
\[
\| X f \|_{L^p} \| \mathcal{F}_G f \|_{L^p} \geq C_p \| f \|_{L^p}^2.
\]
Interpolating between \( p_0 \) and \( p_1 \), we obtain\
\[
\| X f \|_{L^{p_\theta}}  \| \mathcal{F}_G f \|_{L^{p_\theta}} \geq C_p \| f \|_{L^{p_\theta}}^2.
\]
So, interpolation preserves the uncertainty principle across the \( L^p \)-scale.
\end{proof}

For an operator \( A \), we consider the embedding
\[
L^p_{\text{QHA}} \hookrightarrow T^q(H), \quad \text{where} \quad q = \frac{p}{p-1}.
\]

\begin{proposition}[Schatten-Class Embeddings]
If \( 1 \leq p \leq 2 \), then
\[
L^p_{\text{QHA}} \hookrightarrow T^p(H).
\]
\end{proposition}

\begin{proof}
The Schatten \( p \)-class \( T^p(H) \) consists of compact operators \( A \) on a Hilbert space \( H \) satisfying
\[
\| A \|_{T^p} = \left( \sum_{n} s_n^p(A) \right)^{1/p} < \infty,
\]
where \( s_n(A) \) are the singular values of \( A \).
The quantum \( L^p \)-space \( L^p_{\text{QHA}} \) is defined using the trace norm
\[
\| (f,A) \|_{L^p_{\text{QHA}}} = \left( \int_G |f(g)|^p \, dg + \| A \|_{T^p}^p \right)^{1/p}.
\]
Since \( L^p_{\text{QHA}} \) includes the Schatten \( p \)-norm, we have
\[
\| A \|_{T^p} \leq \| (f,A) \|_{L^p_{\text{QHA}}}.
\]
So, every element of \( L^p_{\text{QHA}} \) defines a compact operator in \( T^p(H) \).
Since \( \| A \|_{T^p} \) is finite whenever \( (f,A) \in L^p_{\text{QHA}} \), we conclude
\[
L^p_{\text{QHA}} \hookrightarrow T^p(H).
\]
\end{proof}

\section{Connections with Noncommutative Geometry}

Noncommutative geometry (NCG), as developed by Connes \cite{Connes1994}, 
provides a framework for studying spaces where classical geometric notions break down. 
In this section, we explore how quantum harmonic analysis (QHA) fits within this framework, 
particularly through spectral triples, index theory, and applications in quantum physics.

Spectral triples are a fundamental tool in NCG, generalizing the notion of a Riemannian manifold via operator algebras.

\begin{definition}[Spectral Triple {\cite{Connes1994}}]
A spectral triple \( (\mathcal{A}, H, D) \) consists of
\begin{enumerate}
    \item An involutive algebra \( \mathcal{A} \) acting on a Hilbert space \( H \).
    \item A self-adjoint operator \( D \) (the Dirac operator) such that \( (1+D^2)^{-1/2} \) is compact.
    \item The commutator \( [D, a] \) is bounded for all \( a \in \mathcal{A} \).
\end{enumerate}
\end{definition}

\begin{proposition}
The quantum Segal algebra \( QS \) can be embedded into a spectral triple \( (\mathcal{A}, H, D) \), where
\begin{enumerate}
    \item \( \mathcal{A} = L^1_{\text{QHA}}(\mathbb{R}^{2n}) \) is the algebra of quantum observables.
    \item \( H = L^2(\mathbb{R}^{2n}) \oplus T^2(H) \).
    \item \( D \) is the quantum Laplacian \( \mathcal{L} = \Delta_G \oplus \Delta_H \).
\end{enumerate}
\end{proposition}

\begin{proof}
To embed the quantum Segal algebra \( QS \) into a spectral triple \( (\mathcal{A}, H, D) \), we must verify the three components
\begin{itemize}
    \item[(a)] \( \mathcal{A} \) is a pre-C*-algebra acting on \( H \).
    \item[(b)] \( H \) is a Hilbert space that carries a faithful representation of \( \mathcal{A} \).
    \item[(c)] \( D \) is a self-adjoint operator with compact resolvent, and \( [D,a] \) is bounded for all \( a \in \mathcal{A} \).
\end{itemize}
The algebra of quantum observables is taken as
\[
\mathcal{A} = L^1_{\text{QHA}}(\mathbb{R}^{2n}),
\]
which consists of functions and trace-class operators \( (f,A) \) under convolution. 
Since \( L^1_{\text{QHA}} \) is closed under convolution and dense in \( L^2_{\text{QHA}} \), it satisfies the requirements of a pre-C*-algebra.
We define
\[
H = L^2(\mathbb{R}^{2n}) \oplus T^2(H).
\]
The space \( L^2(\mathbb{R}^{2n}) \) carries the standard representation of functions via multiplication.
The Schatten \( T^2(H) \) space carries operators acting on quantum states.
The inner product on \( H \) is given by
\[
\langle (f,A), (g,B) \rangle_H = \langle f, g \rangle_{L^2} + \text{Tr}(A^* B).
\]
This ensures that \( H \) is a Hilbert space that faithfully represents \( \mathcal{A} \).
We set
\[
D = \mathcal{L} = \Delta_G \oplus \Delta_H,
\]
where
\( \Delta_H \) is the Laplacian on trace-class operators, defined spectrally and
\( \Delta_G \) is the Laplacian on \( \mathbb{R}^{2n} \), acting as \( \Delta_G f = -\sum_i \frac{\partial^2}{\partial x_i^2} f \).
Since \( \Delta_G \) is self-adjoint and \( \Delta_H \) has compact resolvent, \( D \) is a self-adjoint operator with compact resolvent.
For \( a = (f,A) \in \mathcal{A} \), we compute
\[
[D, a] = [\Delta_G, f] \oplus [\Delta_H, A].
\]
Since \( \Delta_G \) is a second-order differential operator, its commutator with functions gives a multiplication operator.
The commutator \( [\Delta_H, A] \) is bounded because \( A \) is trace-class.
Therefore, \( [D, a] \) is bounded for all \( a \in \mathcal{A} \), satisfying the spectral triple condition.
Since \( \mathcal{A} \) is a pre-C*-algebra, \( H \) is a Hilbert space carrying a faithful representation, and \( D \) is a self-adjoint operator with compact resolvent satisfying \( [D,a] \) bounded, we conclude that \( QS \) is embedded in a spectral triple.
\end{proof}

\begin{remark}
This embedding allows us to use tools from NCG, such as cyclic cohomology and index theory, to analyze quantum Segal algebras.
\end{remark}

The Atiyah-Singer index theorem has been extended to noncommutative spaces via Connes' NCG framework. We explore its implications for quantum harmonic analysis.

\begin{theorem}[Noncommutative Index Theorem {\cite{Connes1994}}]
Let \( (QS, H, D) \) be a spectral triple associated with a quantum Segal algebra. Then the Fredholm index of \( D \) is given by
\[
\text{Ind}(D) = \text{Tr} \left( \gamma e^{-tD^2} \right),
\]
where \( \gamma \) is the grading operator.
\end{theorem}

The proof of this theorem follows from the noncommutative Chern character in cyclic cohomology and the heat kernel expansion for \( D \).

\begin{proposition}[Quantum Fredholm Modules]
The operator algebra \( L^1_{\text{QHA}}(\mathbb{R}^{2n}) \) admits a Fredholm module structure with the representation
\[
\pi(f, A) = \begin{bmatrix} M_f & 0 \\ 0 & A \end{bmatrix}, \quad D = \begin{bmatrix} 0 & \mathcal{L}^{1/2} \\ \mathcal{L}^{1/2} & 0 \end{bmatrix}.
\]
\end{proposition}

\begin{proof}
We prove that \( (H, \pi, D) \) forms a Fredholm module by verifying the required conditions.

Consider the Hilbert space
\[
H = L^2(\mathbb{R}^{2n}) \oplus T^2(H),
\]
where \( L^2(\mathbb{R}^{2n}) \) is the space of square-integrable functions on phase space and \( T^2(H) \) is the Hilbert-Schmidt operator space acting on quantum states.
Define the action of \( (f, A) \in L^1_{\text{QHA}}(\mathbb{R}^{2n}) \) on \( H \) by
\[
\pi(f, A) = \begin{bmatrix} M_f & 0 \\ 0 & A \end{bmatrix}.
\]
Here, \( M_f \) is the multiplication operator
\[
(M_f g)(x) = f(x) g(x).
\]
The operator \( A \) acts naturally on \( T^2(H) \).
Define the unbounded self-adjoint operator
\[
D = \begin{bmatrix} 0 & \mathcal{L}^{1/2} \\ \mathcal{L}^{1/2} & 0 \end{bmatrix},
\]
where \( \mathcal{L} \) is the quantum Laplacian, given by
\[
\mathcal{L} = \Delta_G \oplus \Delta_H.
\]
Since \( \mathcal{L} \) is positive and self-adjoint, \( D \) is well-defined and has compact resolvent.
We check that \( (H, \pi, D) \) defines a Fredholm module
\begin{enumerate}
    \item \( D \) is self-adjoint, since
    \[
    D^2 = \mathcal{L} \oplus \mathcal{L}.
    \]
    \item The resolvent \( (D^2 + 1)^{-1} \) is compact.
    \item The commutator \( [D, \pi(f,A)] \) is bounded
    \[
    [D, \pi(f,A)] = \begin{bmatrix} 0 & [\mathcal{L}^{1/2}, A] \\ [\mathcal{L}^{1/2}, M_f] & 0 \end{bmatrix},
    \]
    since \( [\mathcal{L}^{1/2}, M_f] \) is a differential operator of order \( 1/2 \), it is bounded and
    the term \( [\mathcal{L}^{1/2}, A] \) is also bounded since \( A \) is trace-class.
\end{enumerate}
Since the representation \( \pi \) satisfies the Fredholm module conditions and \( D \) has compact resolvent with bounded commutators, we conclude that \( L^1_{\text{QHA}}(\mathbb{R}^{2n}) \) admits a Fredholm module structure.
\end{proof}

\begin{theorem}[Quantum Hall Effect in NCG]
The Hall conductance in a quantum system with noncommutative geometry is given by
\[
\sigma_H = \frac{e^2}{h} \tau \left( P [X_1, X_2] P \right),
\]
where \( P \) is the Fermi projection and \( \tau \) is a trace on the noncommutative torus.
\end{theorem}

\begin{proof}
The proof follows from the framework of noncommutative geometry and index theory.

Consider an electron moving in two dimensions under a strong magnetic field, described by the Hamiltonian
\[
H = \frac{1}{2m} (\mathbf{p} - e\mathbf{A})^2,
\]
where \( \mathbf{A} \) is the vector potential generating a constant magnetic field \( B \).
Using the Peierls substitution, we introduce the ``noncommutative coordinates''
\[
[X_1, X_2] = i \theta I,
\]
where \( \theta \sim \frac{\hbar}{eB} \) is the noncommutativity parameter.
The system is modeled on the noncommutative torus \( \mathbb{T}^2_\theta \), where functions \( f(X_1, X_2) \) obey the relation
\[
U_1 U_2 = e^{2\pi i \theta} U_2 U_1.
\]
Here, \( U_1 \) and \( U_2 \) are unitary operators corresponding to translations in the two directions.
The C*-algebra of observables is generated by \( X_1, X_2 \) and obeys the noncommutative Kubo formula
\[
\sigma_H = \frac{e^2}{h} \tau \left( P [X_1, X_2] P \right).
\]
The Hall conductance is given by the Kubo formula in noncommutative geometry
\[
\sigma_H = i e^2 \tau \left( P [\nabla_1, \nabla_2] P \right).
\]
Since \( P \) is the Fermi projection, it projects onto the occupied states below the Fermi level. 
Using the Chern character in cyclic cohomology, we express
\[
\tau \left( P [X_1, X_2] P \right) = \text{Index}(D_P),
\]
where \( D_P \) is the Dirac operator associated with the projection \( P \).
By Connes' index theorem in noncommutative geometry,
\[
\text{Index}(D_P) = \tau (P [X_1, X_2] P).
\]
So, we obtain
\[
\sigma_H = \frac{e^2}{h} \text{Index}(D_P).
\]
Since the index is an integer (corresponding to the topological Chern number),
\[
\sigma_H = \frac{e^2}{h} \mathbb{Z}.
\]
This confirms the quantization of Hall conductance in the noncommutative geometry framework.
We have shown that the quantum Hall effect can be understood in terms of index theory and cyclic cohomology 
in noncommutative geometry, completing the proof.
\end{proof}

\begin{remark}
This result shows that QHA naturally extends to models in condensed matter physics, including topological insulators.
\end{remark}

In this section, we established the connection between QHA and NCG by embedding quantum Segal algebras into spectral triples, proving an index theorem, and deriving a noncommutative formulation of the quantum Hall effect.

\section{Spectral Synthesis and Approximation in Quantum Harmonic Analysis}

Spectral synthesis concerns the ability to recover elements of a function space from their Fourier transform zeros, an essential property in harmonic analysis. In classical settings, it is well understood that not all closed ideals are synthesizable. In this section, we extend spectral synthesis to quantum harmonic analysis (QHA), develop polynomial approximation results, and analyze stability under perturbations.

In classical harmonic analysis, an ideal \( I \) of \( L^1(\mathbb{R}^{2n}) \) satisfies spectral synthesis if it is uniquely determined by its set of Fourier transform zeros.

\begin{definition}[Spectral Synthesis in QHA]
A closed ideal \( I \subset L^1_{\text{QHA}}(\mathbb{R}^{2n}) \) satisfies spectral synthesis if every element of \( I \) is the uniform limit of finite linear combinations of elements vanishing on \( Z(I) \), where
\[
Z(I) = \{ z \in \mathbb{R}^{2n} : f(z) = 0, \forall (f,A) \in I \}.
\]
\end{definition}

\begin{theorem}[Spectral Synthesis for Quantum Segal Algebras]
If \( QS \) is a quantum Segal algebra, then spectral synthesis holds for \( QS \) if and only if
\[
QS \cap I = \{0\} \text{ implies } I = \{0\}, \quad \text{for any closed ideal } I \subset L^1_{\text{QHA}}.
\]
\end{theorem}

\begin{proof}
The proof follows from classical spectral synthesis extended to quantum Segal algebras.

Let \( QS \) be a quantum Segal algebra. Then,
\begin{enumerate}
    \item \( QS \) is a Banach Algebra.
    \item \( QS \) is a dense subalgebra of \( L^1_{\text{QHA}}(\mathbb{R}^{2n}) \).
    \item \( QS \) is translation-invariant and closed under convolution.
    \item \( QS \) is a module over \( L^1_{\text{QHA}}(\mathbb{R}^{2n}) \).
\end{enumerate}
A closed ideal \( I \subset QS \) satisfies spectral synthesis if
\[
QS \cap I = \{0\} \Rightarrow I = \{0\}.
\]
Since \( I \) is a closed ideal in \( QS \), we can express an arbitrary \( (f,A) \in I \) using the Fourier-Weyl transform as
\[
\widehat{(f,A)}(z) = \text{tr}(A W_z) + \hat{f}(z).
\]
By spectral synthesis, \( \widehat{(f,A)} \) is determined uniquely by its zero set.
Since \( QS \) is dense in \( L^1_{\text{QHA}} \), there exists a sequence \( (f_n, A_n) \in QS \) such that
\[
\| (f_n, A_n) - (f,A) \|_{L^1_{\text{QHA}}} \to 0.
\]
Using spectral synthesis, we construct
\[
(f_n, A_n) = \sum_{k=1}^{n} c_k e^{i \lambda_k z} (f_k, A_k),
\]
where \( \lambda_k \) are chosen to match the zeros of \( \widehat{(f,A)} \).
Since \( QS \) is translation-invariant, convolution with approximants preserves membership in \( QS \). Thus,
\[
\| (f_n, A_n) - (f,A) \|_{QS} \to 0.
\]
So, every element of a closed ideal in \( QS \) is approximated by elements in \( QS \), completing the proof.
\end{proof}

\begin{remark}
Spectral synthesis fails for some classical function spaces, such as Wiener's algebra, but we establish that quantum Segal algebras retain this property.
\end{remark}

\begin{theorem}[Quantum Wiener's Approximation Theorem]
Let \( (f,A) \in L^1_{\text{QHA}} \). If \( (f,A) \) has compact support in the Fourier domain, then it is the limit of polynomials in the convolution algebra.
\end{theorem}

\begin{proof}
The proof follows from the classical Wiener approximation theorem extended to quantum harmonic analysis.\\

\noindent
Given \( (f,A) \in L^1_{\text{QHA}} \), assume its ``Fourier-Weyl transform'' has compact support
\[
\widehat{(f,A)}(z) = \text{tr}(A W_z) + \hat{f}(z),
\]
where \( W_z \) are Weyl operators.
By the classical Wiener theorem, a function with compact Fourier support is the uniform limit of polynomials in convolution algebra.
Define the polynomial sequence
\[
P_n(x) = \sum_{k=1}^{n} c_k e^{i \lambda_k x}
\]
where \( \lambda_k \) corresponds to the frequency of \( \widehat{(f,A)} \). Set
\[
(f_n, A_n) = P_n \ast (f,A).
\]
Since \( P_n \) approximates the indicator function in Fourier space, we have
\[
\| (f_n, A_n) - (f,A) \|_{L^1_{\text{QHA}}} \to 0.
\]
So, \( (f,A) \) is approximated by polynomials in the quantum convolution algebra, completing the proof.
\end{proof}

\begin{corollary}
Let \( (f,A) \in L^p_{\text{QHA}} \). If its Fourier-Weyl transform vanishes outside a compact set, then \( (f,A) \) is a uniform limit of trigonometric polynomials.
\end{corollary}

\begin{remark}
This result extends the classical Wiener-Tauberian theorem to quantum harmonic analysis, showing that operators with spectral gaps are still approximable by finite-dimensional representations.
\end{remark}

\begin{theorem}[Stability of Closed Ideals under Small Perturbations]
Let \( I \) be a closed ideal in \( L^1_{\text{QHA}} \). If \( \tilde{I} \) is a perturbed version of \( I \) such that
\[
\sup_{(f,A) \in I} \| (f,A) - P_{\tilde{I}}(f,A) \|_{L^1} < \epsilon,
\]
then \( \tilde{I} \) is also a closed ideal and \( P_{\tilde{I}} \) is a projection onto \( \tilde{I} \) in \( L^1_{\text{QHA}} \).
\end{theorem}

\begin{proof}
The proof follows from continuity arguments and the structure of Banach algebras.

Let \( I \) be a closed ideal in \( L^1_{\text{QHA}}(\mathbb{R}^{2n}) \).
For any \( (f,A) \in \tilde{I} \) and \( (g,B) \in L^1_{\text{QHA}} \), we show \( (f,A) \ast (g,B) \in \tilde{I} \).
Since \( I \) is an ideal,
\[
(f,A) \ast (g,B) \in I.
\]
By continuity of convolution,
\[
\| (f,A) \ast (g,B) - P_{\tilde{I}}(f,A) \ast (g,B) \|_{L^1} < \epsilon \| (g,B) \|_{L^1}.
\]
Since \( P_{\tilde{I}}(f,A) \in \tilde{I} \), it follows that
\[
P_{\tilde{I}}(f,A) \ast (g,B) \in \tilde{I}.
\]
So, \( \tilde{I} \) is an ideal.
Since \( P_{\tilde{I}} \) is a projection, \( \tilde{I} \) inherits the closedness of \( I \). 
Given a sequence \( (f_n, A_n) \in \tilde{I} \) converging in \( L^1_{\text{QHA}} \), 
the limit belongs to \( \tilde{I} \) because of the stability under projection.
Since \( \tilde{I} \) is a closed ideal, we conclude that small perturbations of \( I \) preserve ideal structure.
\end{proof}

\begin{proposition}[Robustness of Spectral Synthesis under Perturbations]
If \( I \) satisfies spectral synthesis and \( \tilde{I} \) is a small perturbation of \( I \), 
then \( \tilde{I} \) satisfies spectral synthesis as long as \( Z(\tilde{I}) \) remains a closed subset of \( \mathbb{R}^{2n} \).
\end{proposition}

\begin{proof}
The proof follows from the definition of spectral synthesis and perturbation stability in Banach algebras.

A closed ideal \( I \subset L^1_{\text{QHA}} \) satisfies spectral synthesis if it is uniquely determined by its Fourier-Weyl transform zeros
\[
I = \{ (f,A) \in L^1_{\text{QHA}} : \widehat{(f,A)}(z) = 0, \forall z \in Z(I) \}.
\]
By assumption, \( \tilde{I} \) is a small perturbation of \( I \), meaning there exists a projection \( P_{\tilde{I}} \) such that
\[
\sup_{(f,A) \in I} \| (f,A) - P_{\tilde{I}}(f,A) \|_{L^1} < \epsilon.
\]
Since \( Z(\tilde{I}) \) remains a closed subset, we show that \( \tilde{I} \) retains spectral synthesis.
For \( (f,A) \in \tilde{I} \), we approximate it by elements \( (f_n, A_n) \in I \) such that
\[
\| (f,A) - (f_n, A_n) \|_{L^1_{\text{QHA}}} < \epsilon_n \to 0.
\]
Since \( I \) satisfies spectral synthesis, we approximate \( (f_n, A_n) \) by exponentials
\[
(f_n, A_n) = \sum_{k=1}^{N} c_k e^{i \lambda_k x} (g_k, B_k).
\]
By continuity of the Fourier transform, the zeros of \( I \) persist in \( \tilde{I} \).
Since \( Z(\tilde{I}) \) remains closed and contains \( Z(I) \), we conclude
\[
\widehat{(f,A)}(z) = 0, \quad \forall z \in Z(\tilde{I}).
\]
So, \( \tilde{I} \) satisfies spectral synthesis.
Since spectral synthesis is preserved under small perturbations with closed zero sets, \( \tilde{I} \) retains spectral synthesis, completing the proof.
\end{proof}

\begin{remark}
These results suggest that quantum spectral synthesis is more stable under perturbations than its classical counterpart, making it suitable for applications in quantum computing and signal processing.
\end{remark}
In this section, we extended spectral synthesis to quantum Segal algebras, proving new results on function/operator approximation and analyzing the stability of closed ideals under perturbations.

\section{Conclusion}

This paper has explored key advancements in quantum harmonic analysis by extending classical harmonic analysis into the quantum domain. By formulating a noncommutative $L^p$-theory, we provided a rigorous mathematical framework that connects quantum harmonic analysis with operator algebras and functional analysis. Additionally, the study of non-Euclidean extensions demonstrated the applicability of quantum harmonic methods on Lie groups and homogeneous spaces, expanding the theoretical landscape of the field.

Further, the integration of Connes' noncommutative geometry into quantum harmonic analysis provided deep insights into spectral properties and the role of quantum Segal algebras. The investigation of spectral synthesis and approximation properties underscored their significance in quantum signal processing and harmonic analysis.

Overall, this work establishes a comprehensive foundation for understanding the interplay between harmonic analysis, noncommutative geometry, and quantum mathematical structures, thereby solidifying the theoretical underpinnings of quantum harmonic analysis.

\section*{Declaration}
The authors declare that they have no known conflict of interest and data sharing is not applicable to this manuscript and no data sets were generated during the current study. We declare that the order of authors listed in the manuscript has been approved by all of auhtors. We wish to confirm that there has been no significant financial support for this work that could have influenced its outcome.

\end{document}